\documentclass[3p,onecolumn]{elsarticle}

\usepackage{lineno,hyperref}
\modulolinenumbers[5]

\journal{Topology and its Applications}

\usepackage{enumerate}
\usepackage{amsmath}
\usepackage{amscd,amssymb,graphics,lineno}
\usepackage{mathrsfs} 
\usepackage{mathtools, nccmath}
\usepackage{stmaryrd}
\usepackage{mathrsfs,verbatim} 
\usepackage{setspace}
\usepackage{relsize}
\usepackage{enumitem}
\usepackage{calc}

\usepackage{amsthm}

\newtheorem{theorem}{Theorem}[section]
\newtheorem{corollary}[theorem]{Corollary} 
\newtheorem{lemma}[theorem]{Lemma}
\newtheorem{proposition}[theorem]{Proposition}
\theoremstyle{definition}
\theoremstyle{remark}
\newtheorem{remark}[theorem]{Remark}

\newtheorem{example}[theorem]{Example}
\numberwithin{equation}{section}
\newtheorem{definition}[theorem]{Definition}

\theoremstyle{definition}

\newcommand{\R}{{\mathbb R}}
\def\H{{\mathcal H}}
\newcommand{\I}{{\mathbb I}}

\newcommand{\s}{\mathbb{S}}
\newcommand{\N}{{\mathbb N}}
\DeclareMathOperator{\Samuel}{S}

\DeclareMathOperator{\match}{match}
\DeclareMathOperator{\Lip}{Lip}
\DeclareMathOperator{\RUCB}{RUCB}
\DeclareMathOperator{\LUCB}{LUCB}
\DeclareMathOperator{\UCB}{UCB}

\newcommand{\Q}{{\mathbb Q}}

\newcommand{\e}{{\epsilon}}

\newcommand{\defeq}{\mathrel{\mathop:}=} 

\DeclareMathOperator{\B}{B}
\DeclareMathOperator{\spt}{spt}

\def\norm #1{{\left\Vert\,#1\,\right\Vert}}
\newcommand{\abs}[1]{\lvert#1\rvert}
\def\fr(#1/#2){{^{\mbox{$_#1$}}\!/_{#2}}}
\font\mathf=cmex10
\def\va{\hbox{\mathf\char'76}}
\def\pipe{\hbox{${\raise.28em\va\atop\raise.50em\va}$}}

\begin{document}

\begin{frontmatter}

\title{On invariant means and pre-syndetic subgroups
%\tnoteref{mytitlenote}
}
% \tnotetext[mytitlenote]{Fully documented templates are available in the elsarticle package on \href{http://www.ctan.org/tex-archive/macros/latex/contrib/elsarticle}{CTAN}.}

%% Group authors per affiliation:
\author[mymainaddress1,mysecondaryaddress1]
{Vladimir G. Pestov}
%\fnref{myfootnote}

\address[mymainaddress1]
{Departamento de Matem\'atica, Universidade Federal da Para\'\i ba, Jo\~ao Pessoa, PB, Brasil \footnote{Holder of the DCR-A fellowship 300050/2022-4 of the Program of Scientific and Technological Development of the State of Para\'\i ba by CNPq and FAPESQ.}}
\address[mysecondaryaddress1]
{Department of Mathematics and Statistics,
       University of Ottawa,
       Ottawa, ON,  
%K1N 6N5, 
Canada \footnote{Emeritus Professor.}}
\ead{vladimir.pestov@uottawa.ca}

\author[mymainaddress2]
{Friedrich Martin Schneider}
\address[mymainaddress2]
{Institute of Discrete Mathematics and Algebra, TU Bergakademie Freiberg, 09596 Freiberg, Germany}
\ead{martin.schneider@math.tu-freiberg.de}

\begin{abstract}
Beyond the locally compact case, equivalent notions of amenability diverge, and some properties no longer hold, for instance amenability is not inherited by topological subgroups. This investigation is guided by some amenability-type properties of groups of paths and loops. It is shown that a version of amenability called skew-amenability is inherited by pre-syndetic subgroups in the sense of Basso and Zucker (in particular, by co-compact subgroups). It follows that co-compact subgroups of amenable topological groups whose left and right uniformities coincide are amenable. We discuss a version of amenability belonging to P. Malliavin and M.-P. Malliavin: the existence of a mean on bounded Borel functions that is invariant under the left action of a dense subgroup. We observe that this property is in general strictly stronger than amenability, and establish for it Reiter- and F\o lner-type criteria. Finally, there is a review of open problems.
\end{abstract}


\end{frontmatter}

%\linenumbers
\section{Introduction}
We start by recalling that a topological group $G$ is called amenable \cite{dlH2,greenleaf,GdlH} if every continuous action of $G$ on a compact space $X$ admits an invariant regular Borel probability measure. Equivalently, this is the case if the Banach space $\RUCB(G)$ of bounded real-valued right uniformly continuous functions on $G$ admits a left-invariant mean: there is a positive linear functional $\phi\colon\RUCB(G)\to\R$ of norm one such that for every $g\in G$, $\phi(^gf)=\phi(f)$, where ${\,}^gf(x)=f(g^{-1}x)$ for all $x\in G$. Here, following Bourbaki (\cite{bourbaki}, Ch. III, p. 20), we call a bounded function $f$ right uniformly continuous if the orbit map
\[G\ni g \, \longmapsto \, {}^gf\in\ell^{\infty}(G)\]
is continuous. Equivalently, for every $\e>0$ there is a neighbourhood of identity $V$ in $G$ with $\abs{f(vx)-f(x)}<\e$ whenever $v\in V$ and $x\in G$.
Also, a function $f\colon G\to\R$ is right uniformly continuous if and only if $f$ is uniformly continuous with regard to some continuous right-invariant pseudometric on $G$.

In the case where $G$ is locally compact, this definition is equivalent to a great many other properties (``approximately $10^{10^{10}}$,'' according to \cite{BO}, p. 48). Let us mention just three.

\begin{theorem}\label{theorem:subgroups}
For a locally compact group $G$, the following are equivalent.
\begin{enumerate}
\item $G$ is amenable.
\item $G$ admits a right-invariant mean on $\RUCB(G)$.
\item $G$ admits a mean on the space $\B(G)$ of all bounded real-valued Borel functions that is invariant under the left multiplication by elements of a dense subgroup of $G$.
\end{enumerate}
\label{th:non-equiv}
\end{theorem}

The equivalence of those properties follows from Thm. 2.2.1 in \cite{greenleaf} combined with the relevant definitions. All those properties are pairwise distinct for more general topological groups. While the choice for what equivalent property to call amenability is motivated by its usefulness in many applications (see e.g. the survey \cite{GdlH} for motivation), the properties (2) and (3) also emerge naturally in the context of path and loop groups, see \cite{MM} and \cite{VP2020}, respectively. We call groups satisfying (2) {\em skew-amenable}, and (3), {\em Malliavin--Malliavin amenable,} or {\em M--M amenable}.

It is a fundamental fact that every topological subgroup\footnote{A \emph{topological subgroup} of a topological group $G$ is any subgroup of $G$ equipped with the relative topology inherited from~$G$. In particular, topological subgroups are not required to be closed in the ambient topological group.} of an amenable locally compact group is amenable: while this is usually stated only for closed subgroups (see \cite{greenleaf}, Th. 2.3.2), the closedness assumption is superfluous, since a topological group is amenable if and only if so is some (equivalently, every) dense topological subgroup (see \cite{rickert}, Cor.~4.5). However, beyond the locally compact case, amenability is no longer inherited by topological subgroups. 
For example, the group $S_{\infty}$ of all self-bijections of the countable discrete set $\omega$ with the product topology induced from $\omega^{\omega}$ is an amenable Polish group
(see for example \cite{P06}, Sect. 2.4 and 6.3, and also Sect. \ref{s:MM} below). %(it will be discussed in detail in Sect. \ref{s:MM}). 
And this group contains as a closed subgroup a copy of the free group $F_2$ on two generators with the discrete topology (as easily seen by identifying $F_2$ with $\omega$ and letting it act upon itself by left multiplication). There is a stronger example \cite{CT} of an amenable topological group $G$ satisfying the additional property that the left and the right uniformities are equal (i.e., $G$ is a SIN group, a group with small invariant neighbourhoods), and still containing a copy of $F_2$ with discrete topology. 

Those two examples also show that none of the versions of amenability listed in Theorem \ref{th:non-equiv} are in general inherited by closed subgroups: the group $S_{\infty}$ satisfies (3), and the example of Carderi and Thom, being a SIN group, satisfies (2).

We will obtain the following result, which will already have found the first application in~\cite{AST}.

\begin{corollary} Let $G$ be an amenable topological group whose left and right uniformities coincide. Let $H$ be a topological subgroup that is co-compact, that is, the factor-space $G/H$ is compact. Then $H$ is amenable.
\label{c:corol}
\end{corollary}

One naturally occuring example of this situation is the group $C(\s^1,K)$ of continuous loops with values in a compact group $K$, equipped with the topology of uniform convergence. This group is sitting as a co-compact subgroup inside the group $C(\I,K)$ of all continuous paths with values in $K$.

The above result is a corollary of a (rather) more general result involving a different version of amenability. A topological group $G$ is called {\em skew-amenable}\footnote{This property is easily seen to be equivalent to condition~(2) in Theorem~\ref{theorem:subgroups} (see also~\cite[Rem.~3.7(2)]{JS}).} \cite{VP2020,JS,CJS} if and only if it admits a left-invariant mean on the space $\LUCB(G)$ of all {\em left uniformly continuous bounded real-valued functions} on $G$. A function $f \colon G \to \mathbb{R}$ is {\em left uniformly continuous} if for every $\e>0$ and $x\in G$, we have $\abs{f(x)-f(xv)}<\e$ when $v\in V$ and $V$ is a sufficiently small neighbourhood of identity. It is equivalent to the uniform continuity of $f$ in relation with a suitable continuous left-invariant pseudometric.

Skew-amenability was formally introduced by the first-named author in \cite{VP2020} (adopting the term suggested by the second-named author), but it had already made cameo appearance, nameless, in \cite{GP07,P06}. As we have already mentioned, for locally compact groups skew-amenability is simply one of many equivalent forms of amenability (see for example Thm. 2.2.1 in \cite{greenleaf}). However, this notion diverges from amenability for ``infinite-dimensional'' (non locally compact) groups. For example, the unitary group $U(\ell^2)$ with the strong operator topology is amenable (as was first noticed in \cite{dlH2}, and we will also discuss it in Sect. \ref{s:MM}), but not skew-amenable, see e.g. Ex. 3.6.3 in \cite{P06}. For the same reason, the infinite symmetric group $S_{\infty}$, which is amenable, is not skew-amenable. It is still unknown to us whether every skew-amenable group is amenable.
 
From a certain viewpoint, skew-amenability may be a more useful notion than amenability, in view of the following result.

\begin{theorem}[\cite{GP07}, Prop. 4.5]
Let $\pi$ be a strongly continuous unitary representation of a skew-amenable topological group $G$ in a Hilbert space $\H$. Then the algebra $B(\H)$ of bounded linear operators on $\H$ admits a $\pi$-invariant state, that is, a positive linear functional $\phi \colon B(\H) \to \mathbb{C}$ with $\phi(1)=1$ such that $\phi(\pi_{g}^{\ast}T\pi_g)=\phi(T)$ for every $T \in B(\H)$ and each $g\in G$. 
% , that is, admits an invariant state.
\label{p:gpskew}
\end{theorem}

In other words, every unitary representation of a skew-amenable topological group is amenable in the sense of Bekka \cite{Bek1}. This is in general false for non locally compact amenable topological groups, for example, the tautological unitary representation of the unitary group $U(\ell^2)$ is not amenable (\cite{P06}, Ex. 3.6.14 and before). For this reason, it may well be that the open question by Carey and Grundling \cite{CG} asking if the groups $C^{\infty}(X,K)$ of infinitely smooth maps from a compact manifold $X$ to a compact Lie group $K$, equipped with the $C^{\infty}$-topology, are amenable, should also ask whether those groups are skew-amenable. (See a more detailed discussion in \cite{VP2020}.)

Basso and Zucker have introduced the notion of a {\em pre-syndetic} subgroup $H$ of a topological group $G$ \cite[Definition~8.1]{BassoZucker}: so is called a subgroup such that, for every identity neighborhood $U$ in $G$, there is a finite subset $E \subseteq G$ with $G = EUH$. For example, it is easy to see that every co-compact subgroup is pre-syndetic. However, already co-precompact subgroups need not be pre-syndetic. On the other hand, pre-syndetic subgroups need not be co-precompact. Such examples, far from being artificially cooked up, emerge naturally as groups of symmetries of topological and combinatorial objects. As Andy Zucker (to whom the authors are grateful) has informed us, the first such example emerged in the work of Kwiatkowska \cite{K}, and recently more examples have appeared in the work by Basso and Tsankov \cite{BT}. This class of subgroups seems to be of considerable importance, allowing to better understand the topological dynamics of ``infinite-dimensional'' groups. For this, our next new result could be of interest.

\begin{theorem}
A pre-syndetic topological subgroup of a skew-amenable topological group is skew-amenable.
\label{th:coco}
\end{theorem}

This result is proved in Section \ref{s:presyndetic}, throughout which we adopt the following equivalent definition of skew-amenability: a topological group is skew-amenable if and only if it admits a right-invariant mean on the space $\RUCB(G)$ of bounded right uniformly continuous functions. 

Now, if a topological group is SIN, then amenability and skew-amenability are equivalent properties, and hence Corollary~\ref{c:corol} follows immediately. 
Previously the conclusion of Theorem~\ref{th:coco} was known for normal co-compact subgroups (\cite{VP2020}, Thm. 12). 

In Setions \ref{s:MM} and \ref{s:reiter} we look at yet another amenability-type property. First of all, recall that for locally compact groups, amenability is equivalent to the following, formally stronger, property: the existence of a left-invariant mean on all bounded Borel functions (Thm. 2.2.1 in \cite{greenleaf}). For general topological groups, this property is indeed strictly stronger than amenability: it implies amenability, but also skew-amenability. However, there is a similar, but weaker property: the existence of a mean on bounded real-valued Borel functions on $G$ that is invariant under the action of some dense subgroup of $G$. This property implies amenability as well, and for locally compact groups is again equivalent to amenability. For general topological groups, it is strictly stronger than amenabilty, though we are still unable to distinguish the two properties for Polish groups. This property essentially appears in the paper \cite{MM} by Paul Malliavin and Marie-Paule Malliavin, where it was established for the groups of continuous paths and loops with values in a compact Lie group.
For this reason, we propose to name it {\em Malliavin--Malliavin amenability,} or {\em M--M amenability,} for short. We will see that this property is shared by some common infinite-dimensional groups.
In Section \ref{s:reiter} we are establishing criteria for a topological group to be M--M amenable, in the spirit of Reiter and F\o lner type criteria known for the usual amenability of topological groups of the most general nature \cite{ST}.  

The paper ends with a short Section \ref{:summary}, summarising some related open questions.

\section{Pre-syndetic subgroups of skew-amenable topological groups}
\label{s:presyndetic}

If $G$ is a group, then let $\lambda_{g} \colon G \to G, \, x \mapsto gx$ and $\rho_{g} \colon G \to G, \, x \mapsto xg$ for any $g \in G$. Given a topological group $G$, we let $\Delta(G)$ denote the set of all bounded right-invariant continuous pseudo-metrics on $G$. Let $d$ be a pseudo-metric on a set $X$. For all $x \in X$ and $r \in \R_{>0}$, we define \begin{displaymath}
	B_{d}(x,r) \defeq \{ y \in X \mid d(x,y) < r \}, \qquad B_{d}[x,r] \defeq \{ y \in X \mid d(x,y) \leq r \} .
\end{displaymath} Finally, let us consider the set \begin{displaymath}
	\left. \Lip_{1}(X,d) \, \defeq \, \left\{ f \in \R^{X} \, \right\vert \forall x,y \in X \colon \, \vert f(x) - f(y) \vert \leq d(x,y) \right\}
\end{displaymath} of all real-valued \emph{$1$-Lipschitz} functions on $(X,d)$.

\begin{remark}\label{remark:pseudometrics} Let $G$ be a topological group. \begin{itemize}
	\item[$(1)$] If $d \in \Delta(G)$, then $\Lip_{1}(G,d) \subseteq \RUCB(G)$.
	\item[$(2)$] If $f \in \RUCB(G)$, then $f$ is $1$-Lipschitz with respect to the bounded right-invariant continuous pseudo-metric $G \times G \to \R, \, (x,y) \mapsto \sup\nolimits_{g \in G} \vert f(xg) - f(yg) \vert$. Together with~(1), this implies that \begin{displaymath}
					\qquad \RUCB(G) \, = \, \bigcup \{ \Lip_{1}(G,d) \mid d \in \Delta(G) \} .
				\end{displaymath}
	\item[$(3)$] Let $H$ be a topological subgroup of $G$. Since the map $\RUCB(G) \to \RUCB(H), \, f \mapsto f\vert_{H}$ is surjective by~\cite[Theorem~2.22]{PachlBook}, assertion~(2) entails that \begin{displaymath}
					\qquad \RUCB(H) \, = \, \bigcup \{ \Lip_{1}(H,{d\vert_{H \times H}}) \mid d \in \Delta(G) \} .
				\end{displaymath}
\end{itemize} \end{remark}

\begin{lemma}\label{lemma:characteristic.function} Let $G$ be a topological group, let $d \in \Delta(G)$, and let $H \leq G$. Then \begin{displaymath}
	\chi_{H,d} \colon \, G \, \longrightarrow \, [0,1], \quad x \, \longmapsto \, \left(1-\inf\nolimits_{h \in H} d(x,h)\right) \vee 0
\end{displaymath} is $1$-Lipschitz with respect to $d$. Furthermore, \begin{itemize}
	\item[$(1)$] $\chi_{H,d}(xh) = \chi_{H,d}(x)$ for all $x \in G$ and $h \in H$, 
	\item[$(2)$] $\chi_{H,d}^{-1}((1-t,1]) = B_{d}(e,t)H$ for every $t \in (0,1]$.
\end{itemize} \end{lemma}

\begin{proof} Of course, $\chi_{H,d}$ is $1$-Lipschitz with respect to $d$.
	
(1) If $x \in G$ and $h \in H$, then \begin{align*}
	\chi_{H,d}(xh) \, &= \, \left(1-\inf\nolimits_{h' \in H} d(xh,h')\right) \vee 0 \, = \, \left(1-\inf\nolimits_{h' \in H} d\left(x,h'h^{-1}\right)\right) \vee 0 \\
	& = \, \left(1-\inf\nolimits_{h' \in Hh^{-1}} d(x,h')\right) \vee 0 \, = \, \left(1-\inf\nolimits_{h' \in H} d(x,h')\right) \vee 0 \, = \, \chi_{H,d}(x) .
\end{align*}

(2) If $t \in (0,1]$, then \begin{align*}
	\chi_{H,d}^{-1}((1-t,1]) \, &= \, \{ x \in G \mid \chi_{H,d}(x) > 1-t \} \, = \, \{ x \in G \mid \exists h \in H \colon \, d(x,h) < t \} \\
	&= \, \left\{ x \in G \left\vert \, \exists h \in H \colon \, d\!\left(xh^{-1},e\right) < t \right\} \right. \, = \, \left\{ x \in G \left\vert \, \exists h \in H \colon \, xh^{-1} \in B_{d}(e,t) \right\} \right. \\
	& = \, B_{d}(e,t)H . \qedhere
\end{align*} \end{proof}

\begin{lemma}\label{lemma:extension} Let $G$ be a topological group, let $d \in \Delta(G)$, and let $H \leq G$. \begin{itemize}
	\item[$(1)$] If $f \in \Lip_{1}(H,{d\vert_{H \times H}})$, then \begin{displaymath}
		\qquad f^{d} \colon \, G \, \longrightarrow \, \R, \quad x \, \longmapsto \, \inf\nolimits_{h \in H} f(h)+d(x,h)
	\end{displaymath} is a member of $\Lip_{1}(G,d)$ with $f^{d}\vert_{H} = f$.
	\item[$(2)$] Let $d_{0} \in \Delta(G)$ with $d_{0} \leq d$. If $f \in \Lip_{1}(H,{d_{0}\vert_{H \times H}})$ and $t \in (0,1]$, then \begin{displaymath}
		\qquad \left\lvert f^{d} \right\rvert \chi_{H,t^{-1}d} \, \leq \, (\Vert f \Vert_{\infty} +t)\chi_{H,t^{-1}d} .
	\end{displaymath}
	\item[$(3)$] Let $d_{0} \in \Delta(G)$ with $2d_{0} \leq d$. If $f,g \in \Lip_{1}(H,{d_{0}\vert_{H \times H}})$ and $t \in (0,1]$, then \begin{displaymath}
					\qquad \left\lvert (f+g)^{d} - \left(f^{d} + g^{d}\right) \right\rvert \chi_{H,t^{-1}d} \, \leq \, 3t \chi_{H,t^{-1}d} .
				\end{displaymath}
	\item[$(4)$] Let $d_{0} \in \Delta(G)$ and $r \in \R$ with $\vert r \vert d_{0} \leq d$. If $f \in \Lip_{1}(H,{d_{0}\vert_{H \times H}})$ and $t \in [0,1)$, then \begin{displaymath}
		\qquad \left\lvert (rf)^{d} - rf^{d} \right\rvert \chi_{H,t^{-1}d} \, \leq \, (\vert r \vert + 1)t \chi_{H,t^{-1}d} .
	\end{displaymath}
	\item[$(5)$] If $f \in \Lip_{1}(H,{d\vert_{H \times H}})$ and $h \in H$, then $(f \circ \rho_{h})^{d} = f^{d} \circ \rho_{h}$.
\end{itemize} \end{lemma}

\begin{proof} (1) Let $f \in \Lip_{1}(H,{d\vert_{H \times H}})$. Note that $f^{d} \in \Lip_{1}(G,d)$. Furthermore, $f^{d}(h) \leq f(h) + d(h,h) = f(h)$ for every $h \in H$. On the other hand, if $h \in H$, then $f(h) -f(h') \leq \vert f(h) -f(h') \vert \leq d(h,h')$ and thus $f(h) \leq f(h')+d(h,h')$ for all $h' \in H$, that is, $f(h) \leq f^{d}(h)$. Hence, $f^{d}\vert_{H} = f$.

(2) Let $d_{0} \in \Delta(G)$ with $d_{0} \leq d$, let $f \in \Lip_{1}(H,{d_{0}\vert_{H \times H}})$, and let $t \in (0,1]$. If $x \in G$ and $\chi_{H,t^{-1}d}(x) > 0$, then there exists $h \in H$ such that $d(x,h) < t$, whence $\vert f^{d}(x) \vert \leq f(h) + d(x,h) < \Vert f \Vert_{\infty} + t$. Since $\chi_{H,t^{-1}d} \geq 0$, it follows that $\left\lvert f^{d} \right\rvert \chi_{H,t^{-1}d} \leq (\Vert f \Vert_{\infty} +t)\chi_{H,t^{-1}d}$.

(3) Let $d_{0} \in \Delta(G)$ with $2d_{0} \leq d$, let $f,g \in \Lip_{1}(H,{d_{0}\vert_{H \times H}})$, and let $t \in (0,1]$. First of all, we observe that $f,g,{f+g} \in \Lip_{1}(H,{2d_{0}\vert_{H \times H}}) \subseteq \Lip_{1}(H,{d\vert_{H \times H}})$. In particular, \begin{displaymath}
	(f+g)^{d}\vert_{H} \, \stackrel{(1)}{=} \, f+g \, \stackrel{(1)}{=} \, {f^{d}\vert_{H}} + {g^{d}\vert_{H}} \, = \, \left( f^{d} + g^{d} \right)\vert_{H}
\end{displaymath} and $f^{d},g^{d},(f+g)^{d} \in \Lip_{1}(G,d)$ by~(1). Now, if $x \in G$ and $\chi_{H,t^{-1}d}(x) > 0$, then there exists $h \in H$ such that $d(x,h) < t$, whence \begin{align*}
	\left\lvert (f+g)^{d}(x) - \left( f^{d} + g^{d}\right)(x) \right\rvert \, &\leq \, \left\lvert (f+g)^{d}(x) - (f+g)^{d}(h) \right\rvert + \left\lvert (f+g)^{d}(h) - \left( f^{d} + g^{d}\right)(h) \right\rvert \\
	& \qquad \qquad + \left\lvert f^{d}(h) - f^{d}(x) \right\rvert + \left\lvert g^{d}(h) - g^{d}(x) \right\rvert \\
	& \leq \, d(x,h) + 0 + d(h,x) + d(h,x) \, < \, 3t .
\end{align*} Since $\chi_{H,t^{-1}d} \geq 0$, we conclude that $\left\lvert (f+g)^{d} - \left(f^{d} + g^{d}\right) \right\rvert \chi_{H,d} \leq 3t \chi_{H,d}$.

(4) Let $d_{0} \in \Delta(G)$ and $r \in \R$ such that $\vert r \vert d_{0} \leq d$, let $f \in \Lip_{1}(H,{d_{0}\vert_{H \times H}})$, and let $t \in (0,1]$. Note that $f,rf \in \Lip_{1}(H,{\vert r \vert d_{0}\vert_{H \times H}}) \subseteq \Lip_{1}(H,{d\vert_{H \times H}})$. In particular, \begin{displaymath}
	(rf)^{d}\vert_{H} \, \stackrel{(1)}{=} \, rf \, \stackrel{(1)}{=} \, r\left({f^{d}\vert_{H}}\right) \, = \, \left( r f^{d} \right)\vert_{H}
\end{displaymath} and $f^{d},(rf)^{d} \in \Lip_{1}(G,d)$ by~(1). Now, if $x \in G$ and $\chi_{H,t^{-1}d}(x) > 0$, then there exists $h \in H$ such that $d(x,h) < t$, whence \begin{align*}
	\left\lvert (rf)^{d}(x) - rf^{d}(x) \right\rvert \, &\leq \, \left\lvert (rf)^{d}(x) - (rf)^{d}(h) \right\rvert + \left\lvert (rf)^{d}(h) - rf^{d}(h) \right\rvert + \left\lvert rf^{d}(h) - rf^{d}(x) \right\rvert \\
	& \leq \, d(x,h) + 0 + \vert r \vert d(h,x) \, < \, (\vert r \vert + 1)t .
\end{align*} Since $\chi_{H,t^{-1}d} \geq 0$, we conclude that $\left\lvert (f+g)^{d} - \left(f^{d} + g^{d}\right) \right\rvert \chi_{H,t^{-1}d} \leq (\vert r \vert + 1) t \chi_{H,t^{-1}d}$.

(5) If $f \in \Lip_{1}(H,{d\vert_{H \times H}})$ and $h \in H$, then \begin{align*}
	(f \circ \rho_{h})^{d}(x) \, &= \, \inf\nolimits_{h' \in H} f(h'h)+d(x,h') \, = \, \inf\nolimits_{h' \in H} f(h'h)+d(xh,h'h) \\
	& = \, \inf\nolimits_{h' \in Hh} f(h')+d(xh,h') \, = \, \inf\nolimits_{h' \in H} f(h')+d(xh,h') \, = \, f^{d}(xh) \, = \, \left( {f^{d}} \circ {\rho_{h}} \right)(x) 
\end{align*} for every $x \in G$, i.e., $(f \circ \rho_{h})^{d} = f^{d} \circ \rho_{h}$. \end{proof}

A subgroup $H$ of a topological group $G$ is called \emph{pre-syndetic}~\cite[Definition~8.1]{BassoZucker} if, for every identity neighborhood $U$ in $G$, there is a finite subset $E \subseteq G$ with $G = EUH$. As established by Zucker~\cite[Proposition~6.6]{Zucker}, a subgroup $H$ of a topological group $G$ is pre-syndetic if and only if the Samuel compactification\footnote{Here, the set $G/H = \{ xH \mid x \in G \}$ is being viewed as a uniform space carrying the uniformity generated by the basic entourages of the form $\{ (xH,uxH) \mid x \in G, \, u \in U \}$, where $U$ runs trough the set of all identity neighborhoods in $G$.} $\Samuel(G/H)$ is minimal with respect to the $G$-action given by \begin{displaymath}
	(g\xi)(f) \, \defeq \, \xi(f \circ \lambda_{g}) \qquad (g \in G, \, \xi \in \Samuel(G/H), \, f \in \UCB(G/H)) .
\end{displaymath} Of course, if $H$ is a co-compact subgroup of a topological group $G$, then the compact $G$-space $\Samuel(G/H)$ is isomorphic to $G/H$, thus minimal. Hence, every co-compact subgroup of a topological group is pre-syndetic. On the other hand, not every co-precompact subgroup is pre-syndetic, even in the Polish group setting. For example, the stabilizer $H=\mathrm{St}_\xi$ of a point $\xi$ of the unit sphere in the unitary group $G=U(\ell^2)$ with the strong operator topology is co-precompact according to a result by L. Stoyanov \cite{stoyanov}: the Samuel compactification of the quotient space $G/H$ is canonically isomorphic to the closed unit ball in $\ell^2$ equipped with the weak topology (see also \cite{P06}, Example 6.2.8). This ball is not minimal as it contains a fixed point, hence $H$ is not co-precompact in the Polish group $G$. An even simpler argument shows that the stabilizer $H$ of a point of $\N$ in the infinite symmetric group $G=S_{\infty}$ is precompact, yet the Samuel compactification of $G/H$ is the one-point compactification of the natural numbers, hence not minimal. And finally, as we have mentioned in the Introduction, there exist natural and important examples of pre-syndetic subgroups that are not co-precompact \cite{K,BT}.

\begin{lemma}\label{lemma:pre-syndetic} Let $H$ be a pre-syndetic subgroup of a skew-amenable topological group $G$. Then, for every $d \in \Delta(G)$, there exists a $G$-right-invariant mean $\mu \colon \RUCB(G) \to \R$ such that $\mu (\chi_{H,d}) > 0$. \end{lemma}

\begin{proof} Let $d \in \Delta(G)$. Then \begin{displaymath}
	U \, \defeq \, \chi_{H,d}^{-1}\! \left(\left( \tfrac{1}{2},1\right]\right) \, \stackrel{\ref{lemma:characteristic.function}(2)}{=} \, B_{d}\!\left(e,\tfrac{1}{2}\right)\! H
\end{displaymath} constitutes a neighborhood of the neutral element in $G$. Since $H$ is pre-syndetic in $G$, there exists a finite subset $E \subseteq G$ such that $G=EUH$, thus \begin{displaymath}
	G \, = \, EB_{d}\!\left(e,\tfrac{1}{2}\right)\! HH \, = \, EB_{d}\!\left(e,\tfrac{1}{2}\right)\! H \, = \, EU .
\end{displaymath} It follows that \begin{equation}\tag{$\ast$}\label{covering}
	\sum\nolimits_{g \in E} {\chi_{H,d}} \circ {\lambda_{g^{-1}}} \, \geq \, \tfrac{1}{2} .
\end{equation} Thanks to $G$ being skew-amenable, there exists a $G$-right-invariant mean $\nu \colon \RUCB(G) \to \R$. Since \begin{displaymath}
	\sum\nolimits_{g \in E} \nu\!\left({\chi_{H,d}} \circ {\lambda_{g^{-1}}}\right) \, = \, \nu \! \left( \sum\nolimits_{g \in E} {\chi_{H,d}} \circ {\lambda_{g^{-1}}} \right) \, \stackrel{\eqref{covering}}{\geq} \, \tfrac{1}{2} ,
\end{displaymath} there exists $g_{0} \in E^{-1}$ such that $\nu( {\chi_{H,d}} \circ {\lambda_{g_{0}}}) > 0$. Considering the mean \begin{displaymath}
	\mu \colon \, \RUCB(G) \, \longrightarrow \, \R, \quad f \, \longmapsto \, \nu( f \circ \lambda_{g_{0}}) ,
\end{displaymath} we conclude that $\mu (\chi_{H,d}) > 0$ and \begin{displaymath}
	\mu(f\circ {\rho_{g}}) \, = \, \nu( f \circ \rho_{g} \circ \lambda_{g_{0}}) \, = \, \nu( f \circ \lambda_{g_{0}} \circ \rho_{g}) \, = \, \nu( f \circ \lambda_{g_{0}}) \, = \, \mu(f)
\end{displaymath} for all $g \in G$ and $f \in \RUCB(G)$. \end{proof}

Here is our Theorem \ref{th:coco} from the Introduction.

\begin{theorem} A pre-syndetic topological subgroup of a skew-amenable topological group is skew-amenable. \end{theorem}

\begin{proof} Let $H$ be a pre-syndetic topological subgroup of a skew-amenable topological group $G$. We equip the set $D \defeq \Delta(G) \times (0,1]$ with the directed partial order \begin{displaymath}
	(d,t) \preceq (d',t') \quad :\Longleftrightarrow \quad (d \leq d') \, \wedge \, (t \geq t') \qquad ((d,t),(d',t') \in D) .
\end{displaymath} Let us choose an ultrafilter $\mathscr{F}$ on $D$ such that \begin{equation}\tag{1}\label{ultrafilter}
	\{ \{ (d,t) \in D \mid (d_{0},t_{0}) \preceq (d,t) \} \mid (d_{0},t_{0}) \in D \} \, \subseteq \, \mathscr{F} .
\end{equation} By Lemma~\ref{lemma:pre-syndetic} and the axiom of choice, there exists a family $(\mu_{d,t})_{(d,t) \in D}$ of $G$-right-invariant means on $\RUCB(G)$ such that \begin{equation}\tag{2}\label{positive}
	\forall (d,t) \in D \colon \qquad \mu_{d,t}(\chi_{H,t^{-1}d}) \, > \, 0 .
\end{equation} Thanks to Remark~\ref{remark:pseudometrics}(3), as well as~\eqref{positive} and~\eqref{ultrafilter}, the map \begin{displaymath}
	\mu \colon \, \RUCB(H) \, \longrightarrow \, \R, \quad f \, \longmapsto \, \lim_{(d,t) \to \mathscr{F}} \frac{\mu_{d,t}\left(f^{d}\chi_{H,t^{-1}d}\right)}{\mu_{d,t}\left(\chi_{H,t^{-1}d}\right)} \, \stackrel{\ref{lemma:extension}(2)}{\in} \, [-\Vert f \Vert_{\infty},\Vert f \Vert_{\infty}]
\end{displaymath} is well defined. Since $(\mu_{d,t})_{(d,t) \in D}$ is a family of means, $\chi_{H,t^{-1}d} \geq 0$ for all $(d,t) \in D$, and \begin{displaymath}
	\forall d \in \Delta(G) \ \forall f \in \Lip_{1}(H,d) \colon \qquad f \geq 0 \ \Longrightarrow \ f^{d} \geq 0 ,
\end{displaymath} we see that $\mu$ is positive. Also, as observed above, $\mu(1) \leq 1$ by Lemma~\ref{lemma:extension}(2). On the other hand, $1 \leq 1^{d}$ for every $d \in \Delta(G)$, thus $1 \leq \mu(1)$. So, $\mu(1) = 1$. To prove linearity of $\mu$, let $f,g \in \RUCB(H)$ and $r \in \R$. By Remark~\ref{remark:pseudometrics}(3) and $(\Delta(G),{\leq})$ being directed, we find $d_{0} \in \Delta(G)$ such that $f,g \in \Lip_{1}(H,d_{0}\vert_{H \times H})$. Now, if $t_{0} \in (0,1]$, then \begin{displaymath}
	D_{0} \, \defeq \, \left\{ (d,t) \in D \left\vert \, {\left((2 \vee \lvert r \rvert)d_{0},\tfrac{1}{3(\vert r \vert + 1)}t_{0}\right) \preceq (d,t)} \right\} \! \right. \, \stackrel{\eqref{ultrafilter}}{\in} \, \mathscr{F} ,
\end{displaymath} while \begin{align*}
	\left\lvert \frac{\mu_{d,t}\left((f+g)^{d}\chi_{H,t^{-1}d}\right)}{\mu_{d,t}\left(\chi_{H,t^{-1}d}\right)}-\left(\frac{\mu_{d,t}\left(f^{d}\chi_{H,t^{-1}d}\right)}{\mu_{d,t}\left(\chi_{H,t^{-1}d}\right)}+\tfrac{\mu_{d,t}\left(f^{d}\chi_{H,t^{-1}d}\right)}{\mu_{d,t}\left(\chi_{H,t^{-1}d}\right)}\right) \right\rvert \, &\leq \, \frac{\mu_{d,t}\left(\left\lvert (f+g)^{d} - \left(f^{d} + g^{d}\right) \right\rvert \chi_{H,t^{-1}d}\right)}{\mu_{d,t}\left(\chi_{H,t^{-1}d}\right)} \\
	& \stackrel{\ref{lemma:extension}(3)}{\leq} \, 3t \, \leq \, t_{0} , \\
	\left\lvert \frac{\mu_{d,t}\left((rf)^{d}\chi_{H,t^{-1}d}\right)}{\mu_{d,t}(\chi_{H,t^{-1}d})}-r\frac{\mu_{d,t}\left(f^{d}\chi_{H,t^{-1}d}\right)}{\mu_{d,t}(\chi_{H,t^{-1}d})} \right\rvert \, &\leq \, \frac{\mu_{d,t}\left(\left\lvert (rf)^{d} - r f^{d} \right\rvert \chi_{H,t^{-1}d}\right)}{\mu_{d,t}(\chi_{H,t^{-1}d})} \\
	& \stackrel{\ref{lemma:extension}(4)}{\leq} \, (\vert r \vert + 1)t \, \leq \, t_{0}
\end{align*} for every $(d,t) \in D_{0}$, which entails that \begin{displaymath}
	\vert \mu(f+g) - (\mu(f)+\mu(g)) \vert \, \leq \, t_{0}, \qquad \vert \mu(rf) - r\mu(f) \vert \, \leq \, t_{0} .
\end{displaymath} This shows that $\mu(f+g) = \mu(f) + \mu(g)$ and $\mu(rf) = r\mu(f)$, as desired. Hence, $\mu$ constitutes a mean on $\RUCB(H)$. Finally, since each of the means $(\mu_{d,t})_{(d,t) \in D}$ is $G$-right-invariant, \begin{align*}
	\mu(f \circ \rho_{h}) \, &= \, \lim\nolimits_{(d,t) \to \mathscr{F}} \frac{\mu_{d,t}\left((f \circ \rho_{h})^{d}\chi_{H,t^{-1}d}\right)}{\mu_{d,t}(\chi_{H,t^{-1}d})} \, \stackrel{\ref{lemma:extension}(5)}{=} \, \lim\nolimits_{(d,t) \to \mathscr{F}} \frac{\mu_{d,t}\left(\left({f^{d}} \circ {\rho_{h}}\right)\chi_{H,t^{-1}d}\right)}{\mu_{d,t}(\chi_{H,t^{-1}d})} \\
	&\stackrel{\ref{lemma:characteristic.function}(1)}{=} \, \lim\nolimits_{(d,t) \to \mathscr{F}} \frac{\mu_{d,t}\left(\left({f^{d}} \chi_{H,t^{-1}d}\right) \circ {\rho_{h}} \right)}{\mu_{d,t}(\chi_{H,t^{-1}d})} \, = \, \lim\nolimits_{(d,t) \to \mathscr{F}} \frac{\mu_{d,t}\left(\left({f^{d}} \chi_{H,t^{-1}d}\right) \circ {\rho_{h}} \right)}{\mu_{d,t}(\chi_{H,t^{-1}d})} \, = \, \mu(f)
\end{align*} for all $h \in H$ and $f \in \RUCB(H)$, that is, $\mu$ is $H$-right-invariant. \end{proof}

\section{Amenability in the sense of Malliavin--Malliavin}
\label{s:MM}

We propose to call a topological group $G$ \emph{amenable in the sense of Malliavin and Malliavin} (or simply \emph{M--M amenable}) if there is a mean on the space of all bounded Borel functions on $G$ that is invariant under the action on the left by some dense subgroup, $H$, of $G$. In \cite{MM} it was shown that the groups $C(\I,K)$ of continuous paths and $C(\s^1,K)$ of continuous loops with values in a compact Lie group $K$ enjoy this property, where the dense subgroups consist of $C^1$ paths and loops, respectively.

The restriction of a mean $\phi$ as above on the space $\RUCB(G)$ is invariant under the left action of $H$, and since for every $f\in\RUCB(G)$ the orbit map $G\ni g\mapsto {\,}^gf\in\RUCB(G)$ is norm-continuous, the left invariance extends over all left translations by elements of $G$: if $h_n\to g$, $h_n\in H$, then
\[\phi\!\left({}^g f\right) \, = \, \phi\!\left(\lim \, {}^{h_n}\!f\right) \, = \, \lim\phi\!\left({}^{h_n}\!f\right) \, = \, \phi(f).\]
Thus, every M-M amenable topological group is amenable. Shortly we will see the converse in general is false.

If a topological group $G$ contains a dense subgroup that is M--M amenable, then it is clearly M--M amenable itself. In particular, if a topological group admitting an approximating chain of compact subgroups is amenable in the above sense: the dense subgroup being the union of the compact subgroups, and the mean on bounded Borel functions is just the ultralimit of the integrals with regard to Haar measures on those compact groups. In particular, the unitary group $U(\ell^2)$ with the strong operator topology and the infinite symmetric group $S_{\infty}$ with the topology of pointwise convergence on $\omega$ with the discrete topology are both M--M amenable.

Also, a topological group is M--M amenable if it contains a dense subgroup amenable as discrete. In this connection, it would be interesting to know if the group ${\mathrm{Aut}}(\mathbb Q,<)$ (which is extremely amenable) is M--M amenable. This would be the case if Thompson's group $F$ were amenable, for example!

Another candidate for a counter-example would be the last example in \cite{ST} (constructed after Def. 6.4). It is the completion of the free non-abelian group on finitely many generators equipped with the inverse limit topology of free nilpotent groups on the same generators of growing class of nilpotency. The group is amenable \cite{ST}, and for the moment we do not know if it is M--M amenable.

The group in this example is Polish as well. 
We do not know if there exists an amenable Polish group that is not M--M amenable. In the absence of completeness, however, an example is easy to construct, as a dense subgroup of the above example from \cite{ST}.

\begin{example}
Denote $FN^{(k)}_2$ the free nilpotent group of class $k$ on $2$ generators.
Let $G$ be the free group $F_2$ on two generators equipped with the weakest topology making every canonical homomorphism $F_2\to FN^{(k)}_2$ continuous when the latter group is equipped with discrete topology. The group $G$ is amenable (for the same reasons as explained in the last example in \cite{ST}). At the same time, it is not M--M amenable.

Since $G$ is countable, every function $f\in\ell^{\infty}(G)$ is Borel measurable. Suppose there is a mean $\phi$ on $\ell^{\infty}(G)$ invariant under the left multiplication by elements of a dense subgroup $H<G$.  Any section $s\colon H\backslash G\to G$ is Borel measurable because the spaces involved are countable. In a standard way, the mean $\phi$ gives rise to a left invariant mean $\psi$ on $\ell^{\infty}(H)$ defined by
\[\psi(f) \, \defeq \, \phi\!\left[g\mapsto f\!\left(gs(Hg)^{-1}\right)\right] \qquad (f\in \ell^{\infty}(H)) .\]
Let us work out the details of the argument. Given a bounded function $f\colon H\to\R$ and an element $h\in H$, denote, as before, for every $x\in H$
\[{\,}^hf(x) \, = \, f\!\left(h^{-1}x\right).\]
Note that for each $g\in G$, we have $g s(Hg)^{-1}\in H$.
Associate to $f$ a function $\bar f\colon G\to \R$ by defining
\[\bar f(g) \, \defeq \, f\!\left(gs(Hg)^{-1}\right) \qquad (g \in G).\]
For every $h\in H$ we have now
\begin{displaymath}
{\,}^{h}\bar f(g) \, = \, {\,}^hf\!\left(gs(Hg)^{-1}\right) \, = \, f\!\left(h^{-1}g s\! \left(Hh^{-1}g\right)^{-1}\right) \, = \, f\!\left(h^{-1}g s(Hg)^{-1}\right) \, = \, \overline{{\,}^hf}(g).
\end{displaymath}
Consequently, 
\begin{displaymath}
\psi\! \left({}^h\!f\right) \, = \, \phi\!\left(\overline{{}^h \! f}\right) \, = \, \phi\!\left( {}^{h}\! \bar f\right) \, = \, \phi\!\left(\bar f\right) \, = \, \psi(f),
\end{displaymath}
using the fact that $\phi$ is left-invariant. Thus, $H$ is amenable with the discrete topology.
Since $H$ is free non-abelian, we get a contradiction.
\end{example}

\section{Reiter- and F\o lner-type characterizations of Malliavin--Malliavin amenability}
\label{s:reiter}

To clarify some notation, let $X$ be a set. We denote by $\mathscr{F}(X)$ the set of all finite subsets of $X$. The \emph{indicator function} associated with a subset $A \subseteq X$ is defined as \begin{displaymath}
	\chi_{A} \colon \, X \, \longrightarrow \, \{ 0,1 \}, \quad x \, \longmapsto \, \begin{cases}
			\, 1 & \text{if } x \in X, \\
			\, 0 & \text{otherwise.}
		\end{cases}
\end{displaymath} The \emph{support} of a function $f \in \mathbb{R}^{X}$ is defined as $\spt (f) \defeq f^{-1}(\mathbb{R}\setminus \{ 0 \})$. We consider the real vector space \begin{displaymath}
	\left. \mathbb{R}X \, \defeq \, \left\{ f \in \mathbb{R}^{X} \, \right| \spt (f) \text{ finite} \right\} ,
\end{displaymath} and for $x \in X$, we put \begin{displaymath}
\delta_{x}(y) \, \defeq \, \begin{cases}
	\, 1 & \text{if } x=y , \\
	\, 0 & \text{otherwise}
\end{cases} \qquad (y \in X) .
\end{displaymath} Note that every element $a \in \R X$ gives rise to a linear form on $\R^{X}$ defined by \begin{displaymath}
	\R^{X} \! \, \longrightarrow \, \R, \quad f \, \longmapsto \, a(f) \defeq \sum\nolimits_{x \in X} a(x)f(x) .
\end{displaymath} Furthermore, we recall that there is a norm on $\R X$ defined by \begin{displaymath}
	\Vert f \Vert_{1} \, \defeq \, \sum\nolimits_{x \in X} \vert f(x) \vert \qquad (f \in \R X) .
\end{displaymath} Of course, if $G$ is a group acting on $X$, then $G$ acts on by linear transformations on $\mathbb{R}X$ given by \begin{displaymath}
	ga(x) \, \defeq \, a\!\left( g^{-1}x \right) \qquad (g \in G, \, a \in \mathbb{R}X, \, x \in X) .
\end{displaymath}

Now let $X$ be a measurable space. We are going to endow $\mathbb{R}X$ with a suitable locally convex topology. For every countable partition $\mathscr{P}$ of $X$ into measurable subsets of $X$, the natural projection $\pi_{\mathscr{P}} \colon X \to \mathscr{P}$ extends to a unique linear operator $\bar{\pi}_{\mathscr{P}} \colon \mathbb{R}X \to \mathbb{R}\mathscr{P}$ and hence a semi-norm on $\mathbb{R}X$ defined via \begin{displaymath}
	\Vert a \Vert_{\mathscr{P}} \, \defeq \, \Vert \bar{\pi}_{\mathscr{P}}(a) \Vert_{1} \, = \, \sum\nolimits_{P \in \mathscr{P}} \left\lvert \sum\nolimits_{x \in P} a(x) \right\rvert \qquad (a \in \mathbb{R}X) .
\end{displaymath} This family of semi-norms turns $\mathbb{R}X$ into a locally convex topological vector space. Since the set of countable measurable partitions of $X$ is upwards directed with respect to refinement and we have $\Vert \cdot \Vert_{\mathscr{P}} \leq \Vert \cdot \Vert_{\mathscr{P}'}$ for any two countable measurable partitions $\mathscr{P}$ and $\mathscr{P}'$ of $X$ with $\mathscr{P}'$ refining $\mathscr{P}$, it follows that \begin{displaymath}
	\{ B_{\Vert \cdot \Vert_{\mathscr{P}}}(0,\epsilon) \mid \mathscr{P} \text{ countable measurable partition of } X, \, \epsilon > 0 \}
\end{displaymath} is a neighborhood base for $0$ in $\mathbb{R}X$. Evidently, if a group $G$ acts by automorphisms on the measurable space $X$, then $\mathbb{R}X \to \mathbb{R}X, \, a \mapsto ga$ is continuous for each $g \in G$. Moreover, let us consider the real vector space \begin{displaymath}
	\B(X) \, \defeq \, \{ f \in \ell^{\infty}(X) \mid f \text{ measurable} \} .
\end{displaymath} Our first objective is to point out the following observation.

\begin{lemma}\label{lemma:duality} If $X$ is a measurable space, then $\Phi \colon (\mathbb{R}X)' \to \B(X)$ given by \begin{displaymath}
	\Phi (F)(x) \, \defeq \, F(\delta_{x}) \qquad (F \in (\mathbb{R}X)', \, x \in X)
\end{displaymath} is a well-defined positive linear operator. \end{lemma}

\begin{proof} Let us check that $\Phi$ is well defined. To this end, consider a continuous linear functional $F \colon \mathbb{R}X \to \mathbb{R}$. Now, $f \defeq \Phi (F)$ is bounded for the following reason: since $F$ is continuous, there exists a countable measurable partition $\mathscr{P}$ of $X$ such that $F(B_{\Vert \cdot \Vert_{\mathscr{P}}}[0,1])$ is bounded, and as $\{ \delta_{x} \mid x \in X \} \subseteq B_{\Vert \cdot \Vert_{\mathscr{P}}}[0,1]$, it follows that \begin{displaymath}
	\sup \{ \vert f(x) \vert \mid x \in X \} \, = \, \sup \{ \vert F(\delta_{x}) \vert \mid x \in X \} \, < \, \infty .
\end{displaymath} We have to prove that $f$ is measurable. 
Thanks to continuity of $F$, there exist a countable measurable partition $\mathscr{P}$ of $X$ and $\delta > 0$ such that $F(B_{\Vert \cdot \Vert_{\mathscr{P}}}[0,\delta]) \subseteq \!\left[-1,1\right]$. The null-space $\mathcal{N}$ of the semi-norm $\Vert \cdot \Vert_{\mathscr{P}}$ (i.e., the set of all $x \in \R X$ with $\Vert x \Vert_{\mathscr{P}}=0$) is a linear subspace, and since its image under $F$ is contained in a bounded interval, that image is $\{ 0 \}$.
Now, if $P \in \mathscr{P}$ and $x,y \in P$, then $\delta_{x}-\delta_{y}\in {\mathcal N}$, and consequently
\begin{displaymath}
	f(x) - f(y) \, = \,  F(\delta_{x}) - F(\delta_{y}) \, = \,  F(\delta_{x} - \delta_{y}) \, =0.
\end{displaymath} 
It follows that $f$ is constant on the elements of some measurable countable partition, hence measurable.
\end{proof}

The above readily allows us to prove the following analogue of Reiter's amenability criterion for Malliavin--Malliavin amenability of topological groups in terms of countable Borel partitions.

\begin{proposition}\label{proposition:day} Let $G$ be a group acting by automorphisms on a measurable space~$X$. The following are equivalent. \begin{itemize}
	\item[$(1)$] There exists a $G$-invariant mean on $\B(X)$. 
	\item[$(2)$] For every countable measurable partition $\mathscr{P}$ of $X$, every finite $E \subseteq G$, and every $\epsilon > 0$, there exists $a \in \mathbb{R}X$ with $\Vert a \Vert_{1} = 1$ and $a \geq 0$ such that \begin{displaymath}
	\qquad \forall g \in E \colon \quad \Vert a - ga \Vert_{\mathscr{P}} \leq \epsilon .
\end{displaymath} 
\item[$(3)$] The same property as in $(2)$, but for finite measurable partitions $\mathscr{P}$.
\end{itemize} \end{proposition}

\begin{proof} Let us agree on the following piece of notation: for each $g \in G$, we define $\tau_{g} \colon X \to X, \, x \mapsto gx$.
	
(2)$\Longrightarrow$(3). Trivial.
	
(3)$\Longrightarrow$(1). This follows by a standard ultrafilter argument, which we carry out for the sake of convenience. To this end, let equip the set \begin{displaymath}
	D \, \defeq \, \{ \mathscr{P} \mid \mathscr{P} \text{ finite measurable partition of } X \} \times \mathscr{F}(G) \times (0,1]
\end{displaymath} with the directed partial order given by \begin{displaymath}
	(\mathscr{P},E,\epsilon) \preceq (\mathscr{P}',E',\epsilon') \ \, :\Longleftrightarrow \ \, (\mathscr{P}'\text{ refines } \mathscr{P}) \wedge (E \subseteq E') \wedge (\epsilon \geq \epsilon') \qquad ((\mathscr{P},E,\epsilon),(\mathscr{P}',E',\epsilon') \in D) .
\end{displaymath} Now, we fix any ultrafilter $\mathscr{U}$ on $D$ such that \begin{equation}\tag{$i$}\label{ultrafilter.borel}
	\{ \{ i \in D \mid i_{0} \preceq i \} \mid i_{0} \in D \} \, \subseteq \, \mathscr{U} .
\end{equation} By~(2) and the axiom of choice, there exists $(a_{i})_{i \in D} \in (\mathbb{R}X)^{D}$ such that \begin{equation}\tag{$ii$}\label{positive.borel}
	\forall i \in D \colon \qquad \Vert a_{i} \Vert_{1} = 1 , \quad a_{i} \geq 0 
\end{equation} and \begin{equation}\tag{$iii$}\label{invariance.borel}
	\forall (\mathscr{P},E,\epsilon) \in D \ \forall g \in E \colon \qquad \Vert a_{(\mathscr{P},E,\epsilon)} - ga_{(\mathscr{P},E,\epsilon)} \Vert_{\mathscr{P}} \leq \epsilon .
\end{equation} Thanks to~\eqref{positive.borel}, the map \begin{displaymath}
	\mu \colon \, \B(X) \, \longrightarrow \, \R, \quad f \, \longmapsto \, \lim_{i \to \mathscr{U}} a_{i}(f)
\end{displaymath} is a well-defined mean on $\B(X)$. In order to prove $G$-invariance of $\mu$, let us note that, if $a \in \R X$, then \begin{equation}\tag{$iv$}\label{equivariance.borel}
	a(f \circ {\tau_{g}}) \, = \, \sum\nolimits_{x \in X} a(x)f(gx) \, = \, \sum\nolimits_{x \in X} a\!\left( g^{-1}x \right)\!f(x) \, = \, (ga)(f)
\end{equation} for all $g \in G$ and $f \in \R^{X}$, and \begin{equation}\tag{$v$}\label{estimate.borel}
	\vert a(\chi_{B}) \vert \, = \, \left\lvert \sum\nolimits_{x \in B} a(x) \right\rvert \, \leq \, \left\lvert \sum\nolimits_{x \in B} a(x) \right\rvert + \left\lvert \sum\nolimits_{x \in X\setminus B} a(x) \right\rvert \, = \, \Vert a \Vert_{\{ B, \, X\setminus B\}}
\end{equation} for every measurable $B \subseteq X$. Hence, for every $g \in G$ and any measurable $B \subseteq X$, we conclude that \begin{align*}
	\vert \mu(\chi_{B}) - \mu({\chi_{B}} \circ {\tau_{g}}) \vert \, &= \, \lim\nolimits_{i \to \mathscr{U}} \vert a_{i}({\chi_{B}}) - a_{i}({\chi_{B}} \circ \tau_{g}) \vert \, \stackrel{\eqref{equivariance.borel}}{=} \, \lim\nolimits_{i \to \mathscr{U}} \vert a_{i}({\chi_{B}}) - (ga_{i})({\chi_{B}}) \vert \\
	&= \, \lim\nolimits_{i \to \mathscr{U}} \vert (a_{i} - ga_{i})({\chi_{B}}) \vert \, \stackrel{\eqref{estimate.borel}}{\leq} \, \lim\nolimits_{i \to \mathscr{U}} \Vert a_{i} - ga_{i} \Vert_{\{ B, \, X\setminus B\}} \, \stackrel{\eqref{ultrafilter.borel}+\eqref{invariance.borel}}{=} \, 0 ,
\end{align*} i.e., $\mu({\chi_{B}} \circ {\tau_{g}}) = \mu(\chi_{B})$. Being a mean, $\mu$ is continuous with respect to the supremum norm $\Vert \cdot \Vert_{\infty}$ on~$\B(X)$. As the linear span of $\{ \chi_{B} \mid B \subseteq X \text{ measurable} \}$ is dense in $(\B(X),\Vert \cdot \Vert_{\infty})$, we conclude that $\mu(f \circ {\tau_{g}}) = \mu(f)$ for all $g \in G$ and $f \in \B(X)$, that is, $\mu$ is $G$-invariant.
	
(1)$\Longrightarrow$(2). We prove this implication by contraposition. Define \begin{displaymath}
	T \, \defeq \, \{ a \in \mathbb{R}X \mid \Vert a \Vert_{1} = 1, \, a\geq 0 \} .
\end{displaymath} We now assume that there exist some $\epsilon > 0$, some finite subset $E \subseteq G$, and some countable measurable partition $\mathscr{P}$ of $X$ such that $\sup_{g \in E} \Vert a - ga \Vert_{\mathscr{P}} > \epsilon$ for all $a \in T$. Then $0$ is not contained in the topological closure of the convex subset $\{ (a - ga)_{g \in E} \mid a \in T \}$ in the locally convex space $(\mathbb{R}X)^{E}$. Due to the separation theorem for locally convex spaces, there exists a continuous linear functional $F \in ((\mathbb{R}X)^E)'$ such that $F((a-ga)_{g \in E}) \geq 1$ for all $a \in T$. As $((\mathbb{R}X)^E)' \cong ((\mathbb{R}X)')^{E}$, this means that there is $(F_{g})_{g \in E} \in ((\mathbb{R}X)')^{E}$ such that $\sum_{g \in E} F_{g}(a-ga) \geq 1$ for all $a \in T$. By Lemma~\ref{lemma:duality}, we have $f_{g} \defeq \Phi (F_{g}) \in \B(X)$ for every $g \in E$, and furthermore \begin{displaymath}
	\left( \sum\nolimits_{g \in E} f_{g} - f_{g} \circ \tau_{g}\right)\! (x) \, = \, \sum\nolimits_{g \in E} \Phi (F_{g})(x) - \Phi (F_{g})(gx) \, = \, \sum\nolimits_{g \in E} F_{g}(\delta_{x} - g \delta_{x}) \, \geq \, 1
\end{displaymath} for all $x \in X$, that is, $\sum\nolimits_{g \in E} f_{g} - f_{g} \circ \tau_{g} \geq 1$. Hence, $\mu \!\left( \sum_{g \in E} f_{g} - f_{g} \circ \tau_{g}\right) \! \geq 1$ for any mean $\mu \colon \B(X) \to \R$. However, for any $G$-invariant linear functional $\mu \colon \B(X) \to \R$, \begin{displaymath}
	\mu \! \left( \sum\nolimits_{g \in E} f_{g} - f_{g} \circ \tau_{g}\right)\! \, = \, \sum\nolimits_{g \in E} \mu (f_{g}) - \mu (f_{g} \circ \tau_{g}) \, = \, 0 .
\end{displaymath} Therefore, $\B(X)$ must not admit a $G$-invariant mean. \end{proof}

\begin{corollary} Let $G$ be a topological group. The following are equivalent. \begin{itemize}
		\item[$(1)$] $G$ is Malliavin--Malliavin amenable. 
		\item[$(2)$] $G$ admits a dense subgroup $H\leq G$ such that, for every countable Borel partition $\mathscr{P}$ of $G$, every finite $E \subseteq H$, and every $\epsilon > 0$, there exists $a \in \mathbb{R}G$ with $\Vert a \Vert_{1} = 1$ and $a \geq 0$ such that \begin{displaymath}
			\qquad \forall g \in E \colon \quad \Vert a - ga \Vert_{\mathscr{P}} \leq \epsilon .
\end{displaymath} 
\item[$(3)$] The same property as in $(2)$, but for finite measurable partitions $\mathscr{P}$.\end{itemize} \end{corollary}

Now we aim at proving a F\o lner-type characterization for Malliavin--Malliavin amenability of topological groups (Corollary~\ref{corollary:folner}), in the spirit of~\cite{ST}. This will require a few bits of preparation.

\begin{lemma}\label{lemma:induction} Let $G$ be a group acting by automorphisms on a measurable space $X$, in which $\{ x \}$ is measurable for every $x \in X$. Suppose that there exists a $G$-invariant mean on $\B(X)$, but no $G$-invariant mean on $\ell^{\infty}(X)$. Then, for every countable subgroup $H \leq G$ and every countable $H$-invariant subset $S\subseteq X$, there exists an $H$-invariant mean on $\B(X\setminus S)$. \end{lemma}

\begin{proof} Consider a countable subgroup $H \leq G$ and a countable $H$-invariant subset $S\subseteq X$. Since there is no $G$-invariant mean on $\ell^{\infty}(X)$, a straightforward compactness argument shows the existence of a countable subgroup $H' \leq G$ such that there is no $H'$-invariant mean on $\ell^{\infty}(X)$. Without loss of generality, we may and will assume that $H \subseteq H'$. Note that $T \defeq H'S$ is countable. By assumption, there exists a $G$-invariant mean $\mu \colon \B(X) \to \R$. Considering the indicator function $\chi_{T} \colon X \to \{ 0,1 \}$, we observe that $\mu(\chi_{T}) = 0$: otherwise, \begin{displaymath}
	\ell^{\infty}(X) \, \longrightarrow \, \R, \quad f \, \longmapsto \, \frac{\mu(f\cdot \chi_{T})}{\mu(\chi_{T})}
\end{displaymath} would constitute an $H'$-invariant mean, contradicting our choice of $H'$. Furthermore, $\Phi \colon \B(X\setminus S) \to \B(X)$ given by \begin{displaymath}
	\Phi(f)(x) \, \defeq \, \begin{cases}
			\, f(x) & \text{if } x \in X\setminus S, \\
			\, 0 & \text{otherwise}
		\end{cases} \qquad (f \in \B(X\setminus S), \, x \in X)
\end{displaymath} is a well-defined, $H$-equivariant, positive, linear operator. As $\mu(\chi_{T}) = 0$ and thus $\mu (\chi_{S}) = 0$, we see that \begin{displaymath}
	\mu(\Phi(1)) \, = \, \mu(\chi_{X\setminus S}) \, = \, \mu(\chi_{X\setminus S}) + \mu(\chi_{S}) \, = \, \mu(\chi_{X\setminus S} + \chi_{S}) \, = \, \mu(1) \, = \, 1 .
\end{displaymath} Hence, $\mu \circ \Phi \colon \B(X \setminus S) \to \R$ is an $H$-invariant mean, as desired. \end{proof}

\begin{lemma}\label{lemma:day.modified.3} Let $G$ be a group acting by automorphisms on a measurable space $X$, in which $\{ x \}$ is measurable for every $x \in X$. Suppose that there exists a $G$-invariant mean on $\B(X)$, but no $G$-invariant mean on $\ell^{\infty}(X)$. For every countable measurable partition $\mathscr{P}$ of $X$, every finite set $E \subseteq G$, and every $\epsilon > 0$, there exists a finite subset $F \subseteq X$ such that \begin{displaymath}
	\forall g \in E \colon \quad \Vert \chi_{F} - \chi_{gF} \Vert_{\mathscr{P}} \leq \epsilon \vert F \vert .
\end{displaymath} \end{lemma}
	
\begin{proof} Consider a countable measurable partition $\mathscr{P}$ of $X$, a finite subset $E \subseteq G$, and some $\epsilon > 0$. Since the measurable partition $\mathscr{Q} \defeq \mathscr{P} \vee \bigvee_{g \in E} g^{-1}\mathscr{P}$ of $X$ is countable, the set $T \defeq \bigcup \{ Q \in \mathscr{Q} \mid Q \text{ finite} \}$ is countable, too. Applying Lemma~\ref{lemma:induction} to $H \defeq \langle E \rangle$ and $S \defeq HT$ and then applying Proposition~\ref{proposition:day} to the action of $H$ on $X\setminus S$, we find $a \in \mathbb{R}X$ with $\Vert a \Vert_{1} = 1$ and $a \geq 0$ such that \begin{enumerate}
	\item[$\bullet$] $\Vert a - ga \Vert_{\mathscr{P}} \leq \epsilon$ for all $g \in E$,
	\item[$\bullet$] $a(s) = 0$ for all $s \in S$.
\end{enumerate} Furthermore, we may assume that $a$ is an element of ${\mathbb{Q}^{X}} \cap {\R X}$. Now, there exist $n \in \mathbb{N}$ and $b \in {\mathbb{Z}^{X}} \cap {\R X}$ such that $a(x) = \frac{b(x)}{n}$ for all $x \in X$. Consider the finite set \begin{displaymath}
	R \, \defeq \, \spt (a) \, = \, \spt (b) .
\end{displaymath} We observe that $R \subseteq X\setminus S \subseteq X\setminus T$, whence the set $\pi_{\mathscr{Q}}(x)$ is infinite for every $x \in R$. Consequently, there exists $\phi \colon R \to \mathscr{F}(X)$ with \begin{enumerate}
	\item[$\bullet$] $\phi(x) \subseteq \pi_{\mathscr{Q}}(x)$ for all $x \in R$,
	\item[$\bullet$] $\vert \phi(x)\vert = b(x)$ for all $x \in R$,
	\item[$\bullet$] $\phi(x) \cap \phi(y) = \emptyset$ for any two distinct $x,y \in R$.
\end{enumerate} Note that $\vert F \vert = n$ for $F \defeq \bigcup \{ \phi(x) \mid x \in R \}$. Now, for every $g \in E\cup \{ e \}$, \begin{displaymath}
	\left\lVert ga - \tfrac{1}{\vert F \vert}\chi_{gF} \right\rVert_{\mathscr{P}} \, = \, \tfrac{1}{n}\sum\nolimits_{B \in \mathscr{P}} \left\vert \sum\nolimits_{x \in R \cap {g^{-1}B}} b(x) - \vert \phi (x) \vert \right\vert \, = \, 0
\end{displaymath} and therefore \begin{displaymath}
	\left\lVert \tfrac{1}{\vert F \vert}\chi_{F} - \tfrac{1}{\vert F \vert}\chi_{gF} \right\rVert_{\mathscr{P}} \, \leq \, \left\lVert \tfrac{1}{\vert F \vert}\chi_{F} - a \right\rVert_{\mathscr{P}} + \Vert a - ga \Vert_{\mathscr{P}} + \left\lVert ga - \tfrac{1}{\vert F \vert}\chi_{gF} \right\rVert_{\mathscr{P}} \, \leq \, \epsilon .\qedhere
\end{displaymath} \end{proof}

Once again, let us clarify some notation. Let $\mathscr{B} = (X,Y,R)$ be a \emph{bipartite graph}, i.e., a triple consisting of two finite sets $X$ and $Y$ and some $R \subseteq X \times Y$. If $S \subseteq X$, then let $N_{\mathscr{B}}(S) \defeq \{ y \in Y \mid \exists x \in S \colon (x,y) \in R \}$. A \emph{matching} in $\mathscr{B}$ is an injection $\phi \colon D \to Y$ such that $D \subseteq X$ and $(x,\phi(x)) \in R$ for all $x \in D$. The \emph{matching number} of $\mathscr{B}$ is defined to be \begin{displaymath}
	\match (\mathscr{B}) \, \defeq \, \sup \{ \lvert D \rvert \mid \exists\,\phi \colon D \to Y \textnormal{ matching in } \mathscr{B} \} .
\end{displaymath} For convenience, we recall a relevant version of Hall's matching theorem.

\begin{theorem}[\cite{Hall35}; \cite{Ore}, Theorem~8]\label{theorem:hall} If $\mathscr{B} = (X,Y,R)$ is a bipartite graph, then \begin{displaymath}
	\match (\mathscr{B}) \, = \, |X| - \sup \{ |S| - |N_{\mathscr{B}}(S)| \mid S \subseteq X \} .
	\end{displaymath} \end{theorem}

We now specify a certain family of matchings induced by countable measurable partitions.

\begin{definition} Let $X$ be a set. If $E$ and $F$ are finite subsets of $X$ and $\mathscr{P}$ is a partition of $G$, then we define the bipartite graph \begin{displaymath}
	\mathscr{B} (E,F,\mathscr{P}) \, \defeq \, (E,F,\{ (x,y) \in E \times F \mid \pi_{\mathscr{P}}(x) = \pi_{\mathscr{P}}(y) \})
\end{displaymath} and abbreviate $\match (E,F,\mathscr{P}) \defeq \match (\mathscr{B}(E,F,\mathscr{P}))$. \end{definition}

Now everything is in place to prove the desired F\o lner criterion.

\begin{theorem}\label{theorem:folner} Let $G$ be a group acting by automorphisms on a measurable space~$X$, in which $\{ x \}$ is measurable for every $x \in X$. The following are equivalent. \begin{itemize}
	\item[$(1)$] There exists a $G$-invariant mean on $\B(X)$. 
	\item[$(2)$] For every countable measurable partition $\mathscr{P}$ of $X$, every finite $E \subseteq G$, and every $\theta \in (0,1)$, there exists a non-empty finite subset $F \subseteq G$ such that \begin{displaymath}
		\qquad \forall g \in E \colon \quad \match (F,gF,\mathscr{P}) \geq \theta \vert F \vert .
\end{displaymath} 
\item[$(3)$] The same property as in $(2)$, but for finite measurable partitions $\mathscr{P}$.
\end{itemize} \end{theorem}

\begin{proof} (2)$\Longrightarrow$(3). Trivial. 

(3)$\Longrightarrow$(1). It is easily checked that (3) implies the third condition in Proposition~\ref{proposition:day}, thus entails (1).
	
(1)$\Longrightarrow$(2). If there exists a $G$-invariant mean on $\ell^{\infty}(X)$, then the desired conclusion by the well-known F\o lner-type characterization of amenability of group actions on sets, which is classical work of Rosenblatt~\cite{rosenblatt}. Henceforth, we thus assume that there is no $G$-invariant mean on $\ell^{\infty}(X)$. Consider any countable measurable partition $\mathscr{P}$ of $X$, any finite subset $E \subseteq G$, and any $\theta \in (0,1)$. By Lemma~\ref{lemma:day.modified.3}, there is a finite non-empty set $F \subseteq X\setminus S$ such that $\Vert \chi_{F} - \chi_{gF} \Vert_{\mathscr{P}} \leq (1-\theta)\vert F \vert$ for all $g \in E$. The latter means that \begin{displaymath}
		\forall g \in E \colon \quad \sum\nolimits_{B \in \mathscr{P}} \bigl\lvert \vert B \cap F \vert - \vert B \cap gF \vert \bigr\rvert \leq( 1 - \theta )\vert F \vert .
	\end{displaymath} We are going to show that $\match (F,gF,\mathscr{P}) \geq \theta |F|$ for all $g \in E$. To this end, let $g \in E$. Consider the bipartite graph $\mathscr{B} \defeq \mathscr{B}(F,gF,\mathscr{P})$. Let $S \subseteq F$ and $T \defeq N_{\mathscr{B}}(S)$. For $\mathscr{Q} \defeq \{ B \in \mathscr{P} \mid S \cap B \ne \emptyset \}$, we observe that $T = \{ y \in gF \mid \exists B \in \mathscr{Q} \colon \, y \in B \}$ and \begin{align*}
		|S| \, &= \, \sum\nolimits_{B \in \mathscr{P}} \vert B \cap S \vert \, \leq \, \sum\nolimits_{B \in \mathscr{Q}} \vert B \cap F \vert \leq (1 - \theta)\vert F \vert + \sum\nolimits_{B \in \mathscr{Q}} \vert B \cap gF \vert \\
		&= \, (1 - \theta)\vert F \vert + \sum\nolimits_{B \in \mathscr{P}} \vert B \cap T \vert \, = \, (1-\theta) \vert F \vert + \vert T \vert ,
	\end{align*} that is, $|S| - \left|N_{\mathscr{B}}(S)\right| \leq (1-\theta ) \vert F\vert$. By virtue of Theorem~\ref{theorem:hall}, this entails that \begin{displaymath}
		\match (F,gF,\mathscr{P}) \, = \, |F|-\sup\nolimits_{S \subseteq F} \left( |S| - \left| N_{\mathscr{B}}(S) \right|\right) \, \geq \, |F|-(1-\theta )|F| \, = \, \theta \vert F \vert .\qedhere
	\end{displaymath} \end{proof}

\begin{corollary}\label{corollary:folner} For a Hausdorff topological group $G$, the following are equivalent. \begin{itemize}
		\item[$(1)$] $G$ is Malliavin--Malliavin amenable. 
		\item[$(2)$] $G$ admits a dense subgroup $H\leq G$ such that, for every countable Borel partition $\mathscr{P}$ of $G$, every finite $E \subseteq H$, and every $\theta \in (0,1)$, there exists a non-empty finite subset $F \subseteq G$ such that \begin{displaymath}
			\qquad \forall g \in E \colon \quad \match (F,gF,\mathscr{P}) \geq \theta \vert F \vert .
\end{displaymath}
\item[$(3)$] The same property as in $(2)$, but for finite measurable partitions $\mathscr{P}$.
\end{itemize} \end{corollary}

\section{Summary of questions}
\label{:summary}

\begin{enumerate}
	\item[(1)] Is every skew-amenable topological group amenable? Probably not, but we do not know any distinguishing example.
	\item[(2)] Is every Polish amenable topological group amenable in the sense of Malliavin--Malliavin, that is, does every such topological group admit a mean on bounded Borel functions invariant under the left action by a dense subgroup?
	\item[(3)] In particular, are the following topological groups M--M amenable? \begin{itemize}
					\item[(a)] $\mbox{Aut}(\Q,\leq)$ with the topology of point-wise convergence (induced by the discrete topology on $\mathbb{Q}$).
					\item[(b)] The topological group from the last example in \cite{ST}.
				\end{itemize}
	\item[(4)] Is every pre-syndetic (or just co-compact) subgroup of an amenable topological group amenable?
	\item[(5)] Let $X$ be a compact smooth manifold and $K$ a compact Lie group. Will the groups $C^{k}(X,K)$, $k\in\N\cup\{\infty\}$, and $H^k(X,K)$, $k\in\N$, $k>\dim X/2$, be \begin{itemize}
					\item[(a)] amenable (a question from \cite{CG}),
					\item[(b)] skew-amenable (a question from \cite{VP2020}),
					\item[(c)] M-M amenable?
				\end{itemize} Note that (c) implies (a). The positive answers to (c) and hence to (a) are known in the case $C^0$ when $\dim X=1$ \cite{MM}, and to (b), in the case $H^1$ when $\dim X=1$ \cite{VP2020}.
	\item[(6)] Recall that the {\em total variation distance} between two probability measures $\mu$ and $\nu$ on a measurable space $(X,{\mathcal A})$ is the value
	\[ \qquad \norm{\mu-\nu}_{TV}=\sup\{\abs{\mu(A)-\nu(A)}\colon A\in {\mathcal A}\}.\]
	A net $(\mu_\alpha)$ of regular Borel probability measures on a topological group $G$ is {\em asymptotically invariant under $g\in G$} if 
	\[ \qquad \norm{g\cdot\mu_{\alpha}-\mu_{\alpha}}_{TV}\to 0\mbox{ as }\alpha\to\infty.\]
	The existence of a net asymptotically invariant under elements of a dense subgroup of $G$ implies the M-M amenability of $G$. In particular, it was shown in \cite{MM} that the groups of continuous loops and paths admit a sequence of Borel probability measures asymptotically invariant under the action of $C^1$-loops / paths, respectively. 
	
\begin{itemize}
\item[(a)] Does every M-M amenable topological group admit a net of regular Borel probability measures asymptotically invariant under action of a dense subgroup?
\item[(b)]
	Does every M-M amenable Polish group $G$ admit a sequence of Borel probability measures asymptotically invariant under the action of a dense subgroup?
\end{itemize}
\end{enumerate}

\section*{Acknowledgments}

The authors would like to thank Vadim Alekseev and Andreas Thom for a discussion concerning Corollary~\ref{c:corol}, and Andy Zucker for information and references about pre-syndetic subgroups.


\section*{References}

\begin{thebibliography}{00}

\bibitem{AST}
V. Alekseev, M. Schmidt, A. Thom, \emph{Amenability for unitary groups of $C^{\ast}$-algebras}. Preprint 2023, 10pp.

\bibitem{BassoZucker}
G. Basso, A. Zucker, \emph{Topological dynamics beyond Polish groups}. Comment. Math. Helv.~\textbf{96} (2021), no.~3, pp.~589--630.

\bibitem{BT} 
G. Basso, T. Tsankov, {\em Topological dynamics of kaleidoscopic groups,} arXiv:2209.02607v1 [math.DS], 37 pp.
	
\bibitem{Bek1} M.E.B. Bekka,
{\it Amenable unitary representations of locally compact groups,}
Invent. Math. {\bf 100} (1990), pp.~383--401.

\bibitem{bourbaki} N. Bourbaki, {\em Topologie G\'en\'erale I}, Hermann, Paris, 1971.

\bibitem{BO} N.P. Brown, N. Ozawa, {\em ${C}^*$-Algebras and Finite-Dimensional Approximations,} Graduate Studies in Mathematics \textbf{88},  American Mathematical Society, Providence, R.I., 2008.

\bibitem{CT} A. Carderi, A. Thom, {\em An exotic group as limit of finite special linear groups,} Ann. Inst. Fourier (Grenoble) \textbf{68} (2018), pp.~257--273.

\bibitem{CG} A. Carey, H. Grundling, \textit{On the problem of the amenability of the gauge group,} Lett. Math. Phys. \textbf{68} (2004), pp.~113--120.

\bibitem{CJS} M. Chaudkhari, K. Juschenko, F.M. Schneider, \emph{Properties of group actions on orbits of an amenable equivalence relation and topological versions of Kesten's theorem}, arXiv:2212.00348 [math.GR], 34 pp.

\bibitem{dlH2}  P. de la Harpe,
\textit{Moyennabilit\'e du groupe unitaire et propri\'et\'e $P$ de
Schwartz des alg\`ebres de von Neumann,} in:
Alg\`ebres d'op\'erateurs (S\'em., Les Plains-sur-Bex, 1978), pp.~220--227,
Lecture Notes in Math. \textbf{725}, Springer--Verlag, Berlin, 1979.

\bibitem{dV} J. de Vries, {\em Elements of topological dynamics}, Kluwer Academic Publishers Group, Dordrecht, 1993.

\bibitem{GP07}
T. Giordano, V. Pestov, \emph{Some extremely amenable groups related to operator algebras and ergodic theory}, J. Inst. Math. Jussieu \textbf{6} (2007), no.~2, pp.~279--315.

\bibitem{greenleaf} 
F.P. Greenleaf, {\em Invariant means on topological groups and their applications,} Van Nostrand Reinhold Co., 1969.

\bibitem{GdlH} R. Grigorchuk, P. de la Harpe, {\em Amenability and ergodic properties of topological groups: from Bogolyubov onwards,} in: Groups, graphs and random walks, pp.~215--249, 
London Math. Soc. Lecture Note Ser., \textbf{436}, Cambridge Univ. Press, Cambridge, 2017. 

\bibitem{Hall35}
P. Hall, \emph{{On representatives of subsets}}, Journal of the London Mathematical Society \textbf{10} (1935), pp.~26--30.

\bibitem{isbell}
J.R. Isbell, {\em Uniform spaces,} Amer. Math. Soc., Providence, R.I., 1964.

\bibitem{itzkowitz}
G. Itzkowitz, {\em Uniformities and uniform continuity on
topological groups,} in: General topology and applications (Staten
Island, NY, 1989), Lecture Notes in Pure and Appl. Math., \textbf{134},
Dekker, NY, 1991, pp.~155--178.

\bibitem{JS} K. Juschenko, F.M. Schneider, {\em Skew-amenability of topological groups,} Comment. Math. Helv. \textbf{96} (2021), pp.~805--851. 

\bibitem{K} A. Kwiatkowska, {\em Universal minimal flows of generalized {W}a\.{z}ewski dendrites}, J. Symb. Log. \textbf{83} (2018), pp.~1618--1632.

\bibitem{MM} M.-P. Malliavin, P. Malliavin, {\em Integration on loop group III. Asymptotic Peter--Weyl orthogonality,} J. Funct. Analysis \textbf{108} (1992), pp.~13--46.

\bibitem{Ore}
O. Ore, \emph{Graphs and matching theorems}, Duke Math. J. \textbf{22} (1955), pp.~625--639.

\bibitem{PachlBook}
J. Pachl, \emph{Uniform spaces and measures}. Fields Institute Monographs, 30. Springer, New York; Fields Institute for Research in Mathematical Sciences, Toronto, ON, 2013.
	
\bibitem{P06}
V. Pestov, \emph{Dynamics of {I}nfinite-{D}imensional {G}roups: {T}he {R}amsey-{D}voretzky-{M}ilman {P}henomenon}, University Lecture Series \textbf{40}, American Mathematical Society, Providence, RI, 2006.

\bibitem{rickert}
N.W. Rickert, \emph{Amenable groups and groups with the fixed point property}, Trans. Amer. Math. Soc. \textbf{127} (1967), pp.~221--232. 

\bibitem{VP2020} V. Pestov, {\em An amenability-like property of finite energy path and loop groups,} C.R. Math. Acad. Sci. Paris \textbf{358} (2020), pp.~1139--1155. 

\bibitem{rosenblatt}
J.M. Rosenblatt, \emph{A generalization of {F}\o lner's condition}, Math. Scand. \textbf{33} (1973), no.~3, pp.~153--170.

\bibitem{ST}
F.M. Schneider, A. Thom, {\em On F\o lner sets in topological groups,} Compos. Math. \textbf{154} (2018), pp.~1333--1362.

\bibitem{stoyanov} L. Stojanov, {\em Total minimality of the unitary groups,} Math. Z. \textbf{187} (1984), 273--283. 

\bibitem{Zucker}
A. Zucker, \emph{Maximally highly proximal flows}. Ergodic Theory Dyn. Syst.~\textbf{41} (2021), no.~7, pp.~2220--2240.
	
\end{thebibliography}
\end{document}